\documentclass[a4paper,fleqn]{article}
\usepackage{a4wide}
\usepackage{authblk}
\providecommand{\keywords}[1]{\smallskip\noindent\textbf{\emph{Keywords:}} #1}

\providecommand{\msc}[1]{\smallskip\noindent\textbf{\emph{2020 Mathematics Subject Classification:}} #1}
\usepackage{ifthen}          
\usepackage{amsmath,amssymb,amsthm}
\usepackage{thmtools}

\usepackage{graphicx}
\usepackage{tikz}
\usetikzlibrary{fit,patterns,decorations.pathmorphing,decorations.pathreplacing,calc}

\usepackage{caption}

\usepackage{multirow}

\usepackage{etoolbox}

\usepackage{hyperref}
\usepackage[capitalise,nameinlink,sort]{cleveref}

\usepackage[T1]{fontenc}
\usepackage[utf8]{inputenc}  

\usepackage{mleftright}       
\mleftright                   

\declaretheorem[numberwithin=section]{theorem}
\declaretheorem[sibling=theorem]{lemma,corollary,proposition}
\declaretheorem[sibling=theorem,style=definition]{example,definition,remark}
\crefname{equation}{}{}

\newcommand{\N}{\mathbb{N}}
\newcommand{\Z}{\mathbb{Z}}
\newcommand{\R}{\mathbb{R}}
\newcommand{\ord}[2]{\mathsf{ord}\left(#1,#2\right)}

\DeclareMathOperator{\sign}{sign}
\DeclareMathOperator{\lcm}{lcm}
\DeclareMathOperator{\graph}{graph}
\DeclareMathOperator{\Card}{\#}

\DeclareMathOperator{\vlog}{\rm log}

\newenvironment{dedication}
  {\thispagestyle{empty}
   \vspace*{.5truecm}
   \itshape             
   \raggedleft          
  }
  {\par 
   \vspace{1truecm} 
  }
\begin{document}

\title{Doubling modulo odd integers, generalizations, and unexpected occurrences}

\author{Jean-Paul Allouche}
\affil{CNRS, IMJ-PRG, Sorbonne Universit\'e, Paris, France\\
\url{jean-paul.allouche@imj-prg.fr}
}

\author{Manon Stipulanti\thanks{Manon Stipulanti is an FNRS Research Associate supported by the FNRS research grant 1.C.104.24F.}}
\affil{Department of Mathematics, University of Li\`{e}ge, Belgium\\
\url{m.stipulanti@uliege.be}
}

\setcounter{footnote}{+2}
\author{Jia-Yan Yao\thanks{Jia-Yan Yao is partially supported by the National Natural Science Foundation of China (Grant No. 12231013).}}
\affil{Department of Mathematics, Tsinghua University, Beijing 100084\\
People's Republic of China,
\url{jyyao@mail.tsinghua.edu.cn}
}

\maketitle            

\begin{dedication}
 A tribute to the memory of Jacques Roubaud 1932--2024   
\end{dedication}

\begin{abstract}
The starting point of this work is an equality between two 
quantities $A$ and $B$ found in the literature, which involve
the {\em doubling-modulo-an-odd-integer} map, i.e., 
$x\in \N \mapsto 2x \bmod{(2n+1)}$ for some positive
integer $n$. More precisely, this doubling map defines a 
permutation $\sigma_{2,n}$ and each of $A$ and $B$ counts 
the number $C_2(n)$ of cycles of $\sigma_{2,n}$, hence $A=B$.
In the first part of this note, we give a direct proof of 
this last equality. To do so, we consider and study a 
generalized $(k,n)$-perfect shuffle permutation 
$\sigma_{k,n}$, where we multiply by an integer $k\ge 2$ 
instead of $2$, and its number $C_k(n)$ of cycles.
The second part of this note lists some of the many 
occurrences and applications of the doubling map and 
its generalizations in the literature: in mathematics 
(combinatorics of words, dynamical systems, number theory,
correcting algorithms), but also in card-shuffling,
juggling, bell-ringing, poetry, and music composition.


\msc{11B50, 11B83, (primary); 05A05, 05A19, 11A25, 20B30, 20B99,
00A65 (secondary)}
\end{abstract}

\section{Introduction}

The starting point of this note was the observation that the permutation $\tau_m$ (for some odd integer $m\ge 3$) occurring in the 2021 paper~\cite{Guo-Han-Wu-2021} by Guo, Han, and Wu, and defined on 
$[0,m-2]$ by
\begin{equation}
\label{eq: permutation tau m}
\tau_m := 
\left(
\begin{matrix}
0 \ &1 &\cdots &\frac{m-3}{2} \ &\frac{m-1}{2} \ 
&\frac{m+1}{2} &\cdots &m-2 \\
1 \ &3 &\cdots &m-2 &0 &2 &\cdots &m-3 \\
\end{matrix}
\right),
\end{equation}
is the same (up to notation) as the permutation $\sigma_n$ (for some integer $n\ge 1$) occurring in the 1983-1984 paper~\cite{Allouche-1983-1984} by Allouche, and defined on $[1,2n]$ by
\begin{equation}
\label{eq: permutation sigma n}
\sigma_n :=
\left(
\begin{matrix}
1 \ &2 &\cdots &n \ &n+1 \ &n+2 \ &\cdots &2n \\
2 \ &4 &\cdots &2n \ &1 \ &3 &\cdots &2n-1 \\
\end{matrix}
\right).
\end{equation}
(For slightly more on these two papers, see~\cref{sec:origin} 
below.) The fact that these permutations are the same up 
to notation can be seen, e.g., from the fact that the first 
one can be defined by $i \to 2i+1 \bmod{m}$ (for each $i \in [0, m-2]$), while the second one can be 
defined by $j \to 2j \bmod{(2n+1)}$ (for each $j \in [1, 2n])$.
In particular, the number of cycles in the decomposition 
into a product of disjoint cycles of $\tau_m$ is given 
in~\cite{Guo-Han-Wu-2021} by
\begin{equation}\label{expression1}
-1 + \frac{1}{\ord{2}{m}} \sum_{j=0}^{\ord{2}{m} - 1} \gcd(2^j-1, m),
\end{equation}
while it is given in~\cite{Allouche-1983-1984} by
\begin{equation}\label{expression2}
\sum_{\substack{d|m \\ d \neq 1}} \frac{\varphi(d)}{\ord{2}{d}},
\end{equation}
where $\varphi$ is the Euler totient function and, for each odd integer $\ell$, $\ord{2}{\ell}$ is the multiplicative order of $2$ modulo $\ell$.
(See precise definitions in the next section.)

The quantities in~\cref{expression1,expression2} are equal 
since they count the same objects. While wanting to give a 
direct proof of their equality, i.e., without associating 
them with counting some objects, we came across many 
occurrences of the innocent-looking map 
\begin{align}
\label{eq: doubling modd odd}
x\in \N \mapsto  2x \bmod{(2n+1)},
\end{align}
where $n$ is a fixed positive integer, as accounted below in~\cref{sec: occurrences}.
The previous map is naturally called \emph{doubling modulo an odd integer}.
The purpose of this note is twofold.
First, we propose direct proofs of the equality between the quantities in~\cref{expression1,expression2}, as a particular case of our~\cref{pro: identity with phi and ord} with~\cref{cor: equality between our two quantities}.
Before doing so, we consider in~\cref{sec: first interpretation} a permutation, noted $\sigma_{k,n}$, that 
generalizes $\sigma_n$ from~\cref{eq: permutation sigma n}, 
and describe the number $C_k(n)$ of cycles in the 
decomposition of $\sigma_{k,n}$ into a product of disjoint 
cycles.
\cref{pro: identity with phi and ord} is then showed 
in~\cref{sec: second interpretation}, for which we give 
two proofs with different flavors.
In~\cref{sec: asymptotics}, we study the asymptotics of the related sequence $(C_k(n))_{n\ge 1}$ for each value of $k$.
Finally, in~\cref{sec: occurrences}, we review some of many manifestations and applications of the doubling-modulo-odd-integers map of~\cref{eq: doubling modd odd} and its generalizations, as well as some numerous occurrences of the quantity defined in~\cref{expression2} that we have found in the literature.

\subsection{Notation}

For all integers $k\ge 2$ and $n\ge 1$, we define the \emph{$(k,n)$-perfect shuffle permutation} to be the permutation $\sigma_{k,n} \colon x \mapsto kx \bmod{(kn+1)}$ for $x\in[1,kn]$ (we take its name after~\cite{Ellis-Fan-Shallit-2002}).
The case $k=2$ corresponds to the permutation $\sigma_n$ from~\cref{eq: permutation sigma n}, which is also called the \emph{perfect shuffle permutation of order $2n$} in~\cite{Ellis-Krahn-Fan-2000}.

\begin{example}
For $k=2$ and $n=3$, we have $\sigma_{2,3} \colon 1\mapsto 2, 2 \mapsto 4, 3 \mapsto 6, 4 \mapsto 1, 5 \mapsto 3, 6 \mapsto 5$.
\end{example}

In particular, we are interested in the decomposition
of $\sigma_{k,n}$ into a product of disjoint cycles.
For all integers $k\ge 2$ and $n\ge 1$, we let $C_k(n)$ be 
the number of cycles in the decomposition of the permutation 
$\sigma_{k,n}$ into a product of disjoint cycles.
By \emph{cycles}, we mean cycles of all lengths, even 
singletons that correspond to elements fixed by the 
permutation.

\begin{example}
\label{ex: cycle decomposition}
The decomposition of $\sigma_{2,3}$ into cycles  gives 
$\sigma_{2,3} = (1 2 4) (3 6 5)$, so $C_2(3)=2$.
The cycles in the decomposition of $\sigma_{3,5}$ are $(1 \; 3 \; 9 \; 11)$, $(2 \; 6)$, $(4 \; 12)$, $(5 \; 15 \; 13 \; 7)$, $(8)$, and $(10 \; 14)$ with respective length $4,2,2,4,1,2$, so $C_3(5)=6$.

For various values of $k$, we have computed in~\cref{tab: first terms of various sequences C_k(n)} the first few terms of the sequence $(C_k(n))_{n\ge 1}$.
The sequence corresponding to $k=2$ is~\cite[A006694]{Sloane}, while the others do not seem to be indexed in Sloane's OEIS~\cite{Sloane}.
\begin{table}[ht]
    \centering
    \[
\begin{array}{c|l}
    k & (C_k(n))_{n\ge 1} \\
    \hline
    2 & 1, 1, 2, 2, 1, 1, 4, 2, 1, 5, 2, 2, 3, 1, 6, 4, 5, 1, 4, 2, \ldots \\
    3 & 2, 1, 3, 4, 6, 1, 5, 2, 6, 1, 3, 2, 12, 1, 5, 2, 14, 5, 3, 6, \ldots \\
    4 & 2, 4, 2, 4, 8, 4, 2, 8, 2, 4, 14, 4, 2, 8, 2, 12, 8, 8, 8, 8, \ldots \\
    5 & 3, 2, 7, 4, 7, 10, 11, 2, 3, 4, 13, 2, 11, 14, 11, 4, 3, 10, 25, 4, \ldots \\
    6 & 3, 1, 2, 8, 5, 9, 14, 6, 9, 1, 2, 2, 1, 9, 10, 8, 1, 1, 14, 2, \ldots \\
    7 & 4, 5, 3, 4, 14, 7, 13, 20, 16, 1, 11, 6, 6, 9, 5, 8, 38, 1, 3, 8, \ldots \\
\end{array}
\]
    \caption{For $k\in [2,7]$, the first few terms of the sequence $(C_k(n))_{n\ge 1}$, where $C_k(n)$ gives the number of cycles in the decomposition of the permutation $\sigma_{k,n}$ into a product of disjoint cycles.}
    \label{tab: first terms of various sequences C_k(n)}
\end{table}
\end{example}

To obtain several descriptions of the quantity $C_k(n)$, we introduce some notation.
For two integers $a,b\ge 1$, we let $\gcd(a,b)$ denote their greatest common divisor.
We also set $\gcd(0,b) = b$ if $b\neq 0$.
For a finite set $A$, we let $\Card A$ denote the number of 
elements of $A$, i.e., 
\begin{align}
    \label{eq: set cardinality}
\Card A = \sum_{k \in A} 1.
\end{align}
We let $\varphi$ denote the Euler totient function defined, for an integer $n\ge 1$, by the number of integers in the interval $[1, n]$ that are coprime with $n$, i.e., $\varphi(n) = \Card \{ k \in [1, n] \mid \gcd(k,n)=1  \}$.
See~\cref{fig: Euler totient} for the first few values of the
Euler totient function, which is also sequence A000010 
in Sloane's OEIS~\cite{Sloane}.
\begin{figure}
    \centering
    \includegraphics[width=0.5\linewidth]{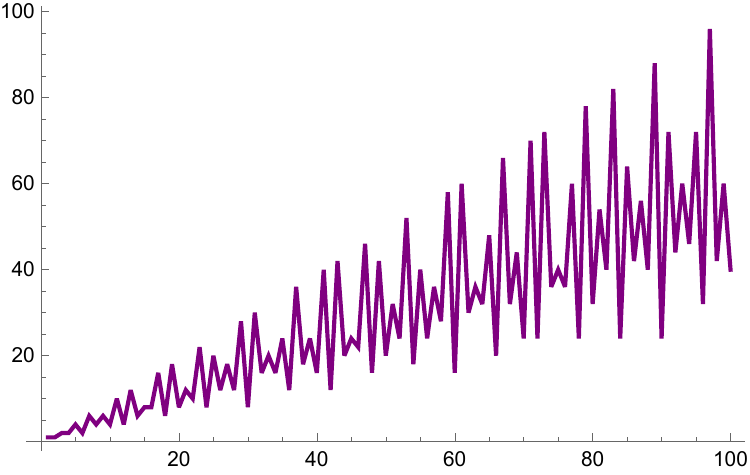}
    \caption{The first few values of Euler's totient function $\varphi$.}
    \label{fig: Euler totient}
\end{figure}
We recall (and reprove) the next well-known formula (e.g., see~\cite[Relation~(33), page~32]{Tenenbaum}) that will be 
useful later on.

\begin{lemma}
\label{lem: famous formula with varphi}
For each integer $n\ge 1$, we have $\sum_{d | n} \varphi(d) = n$.
\end{lemma}
\begin{proof}
Let us quickly recall how to prove this equality. We may proceed as follows
\[
n= \sum_{1 \leq k \leq n} 1 =  
\sum_{d | n} \sum_{\substack{1 \leq k \leq n \\ \gcd(k, n) = d}} 1
= \sum_{d | n} \sum_{\substack{1 \leq k' \leq n/d \\ \gcd(k', n/d) = 1}} 1
= \sum_{d | n} \varphi(n/d) = \sum_{d' | n} \varphi(d'),
\]
where the last equality is obtained via the classical change 
of summation index $d' = n/d$.  
\end{proof}

Given a positive integer $n$ and an integer $a$ coprime to $n$, we let $\ord{a}{n}$ be the \emph{multiplicative order of $a$ modulo $n$}, which is the smallest positive integer $\ell$ such that $a^\ell \equiv 1 \bmod n$.
Note that $\ord{a}{1}=1$ for each integer $a\ge 2$.
For example, the sequence $(\ord{2}{2n+1})_{n\ge 0}$ starts 
with $1, 2, 4, 3, 6, 10, 12$; it is sequence A002326 in 
Sloane's OEIS~\cite{Sloane}.
We let $(\Z/n\Z)^{\times}$ be the multiplicative group of invertible elements of $\Z/n\Z$.
Recall that an integer $k$ is a \emph{primitive root modulo $n$} if $k$ is a generator of the multiplicative group of integers modulo $n$, i.e., if for every integer $a$ coprime to $n$, there is some integer $\ell$ for which $k^\ell \equiv a \bmod{n}$.
In particular, for $k$ to be a primitive root modulo $n$, a necessary and sufficient condition is that
$\ord{k}{n}=\varphi(n)$.

\section{A first interpretation and an incursion into permutation group theory}
\label{sec: first interpretation}

One way of interpreting the quantity $C_k(n)$ is the following one.
Also see~\cite{Allouche-1983-1984,Ellis-Krahn-Fan-2000} for the particular case $k=2$.



\begin{proposition}
\label{pro: number of cycles in sigma kn} 
For all integers $k\ge 2$ and $n\ge 1$, the number $C_k(n)$ 
of cycles in the decomposition of the permutation 
$\sigma_{k,n}$ into a product of disjoint cycles is equal to
\[
\sum_{\substack{d | kn+1 \\ d \neq 1}} \frac{\varphi(d)}{\ord{k}{d}}\cdot
\]
\end{proposition}

\begin{example}
\label{ex: k=3 and n=5}
With $k=3$ and $n=5$, recall from~\cref{ex: cycle decomposition} that the cycles of $\sigma_{k,n}$ are $(1 \; 3 \; 9 \; 11)$, $(2 \; 6)$, $(4 \; 12)$, $(5 \; 15 \; 13 \; 7)$, $(8)$, and $(10 \; 14)$.
Their number is given by
\[
C_k(n)=
\sum_{\substack{d | kn+1 \\ d \neq 1}} \frac{\varphi(d)}{\ord{k}{d}}
= \frac{1}{1} + \frac{2}{2} + \frac{4}{2} + \frac{8}{4} = 6.
\]
\end{example}

\begin{proof}[Proof of~\cref{pro: number of cycles in sigma kn}]
The cardinality of the orbit of $j\in[1,kn]$ under $\sigma_{k,n}$ is
\[
\ord{k}{\frac{kn+1}{\gcd(j, kn+1)}},
\]
which is the smallest integer $s\geq 1$ such that $k^sj\equiv j\bmod (kn+1)$.
(We note that if two elements $i,j$ are in the same orbit, then $\gcd(i, kn+1) = \gcd(j, kn+1)$, since $k$ and $kn+1$ are coprime; this follows from an adaptation of~\cite[Proposition~1]{Ellis-Krahn-Fan-2000} to the general case.)
So, letting $d = \gcd(j, kn+1)$, the cycles in the decomposition of $\sigma_{k,n}$ have length 
\begin{align}
\label{eq: possible cycle length}
    \ord{k}{\frac{kn+1}{d}}, 
\end{align}
where $d$ divides $(kn+1)$ and $d \neq kn+1$ (since $d=\gcd(j,kn+1)$ and $1\leq j\leq kn$). 
Furthermore, the number of cycles with that particular length is
\begin{align}
\label{eq: number of cycles with fixed length}
    \frac{\varphi(\frac{kn+1}{d})}{\ord{k}{\frac{kn+1}{d}}}.
\end{align}
So, the number of cycles in the decomposition of the
permutation $\sigma_{k,n}$ into a product of disjoint 
cycles is equal to
\[
C_k(n) = \sum_{\substack{d | kn+1 \\ d \neq kn+1}} 
\frac{\varphi(\frac{kn+1}{d})}{\ord{k}{\frac{kn+1}{d}}},
\]
which, after the change of the summation index 
$d' = (kn+1)/d$, yields the expected equality.
\end{proof}

Due to the proof of the previous result (in particular~\cref{eq: possible cycle length,eq: number of cycles with fixed length}), we obtain the following information on the cycles of the permutation $\sigma_{k,n}$.

\begin{corollary}
\label{cor: info sur les cycles}
Let $k\ge 2$ and $n\ge 1$ be integers.
For each divisor $d$ of $kn+1$ with $d\neq 1$, the permutation $\sigma_{k,n}$ possesses $\varphi(d)/\ord{k}{d}$ cycles of length $\ord{k}{d}$ in its decomposition into a product of disjoint cycles.
\end{corollary}

In fact, the structure of the cycles in the $(n,k)$-perfect shuffle permutation is precisely described in~\cite[Theorem~1]{Ellis-Fan-Shallit-2002} restated below.
Furthermore, the authors provide a linear-time algorithm to compute representatives of the cycles, which they call \emph{seeds} (the \emph{seed set} contains integers in $[1,kn]$ such that no two seeds are in the same cycle and every cycle contains a seed).

\begin{theorem}[{\cite[Theorem~1]{Ellis-Fan-Shallit-2002}}]
\label{thm: precise description of cycles}
Let $k\ge 2$ and $n\ge 1$ be integers.
The $r$-tuple $(a_0,a_1,\ldots,a_{r-1})$ with $r\ge 1$ is a cycle of the $(n,k)$-perfect shuffle permutation if and only if there exist a divisor $d$ of $kn+1$ with $d\neq 1$ and an element $a\in [1,d-1]$ (coprime with $d$) such that $r=\ord{k}{d}$ and $a_i = \frac{a(kn+1)}{d} k^i \bmod{(kn+1)}$ for $i\in [0,r-1]$. 
\end{theorem}
\begin{proof} 
    We repeat the proof of~\cref{pro: number of cycles in sigma kn}, but with more precision. Let   $(a_0,a_1,\ldots,a_{r-1})$ be a cycle of the $(n,k)$-perfect shuffle permutation with $r\ge 1$.  By definition, this is equivalent to saying that there exists an integer $j\in [1,kn]$ such that $a_i\equiv jk^i\ \bmod (kn+1)$ $(0\leq i<r)$, where $r$ is the smallest integer $s\geq 1$ such that $jk^s\equiv j\bmod (kn+1)$. Put $d=\frac{kn+1}{\gcd(j,kn+1)}$, and write $j=a\gcd(j,kn+1)$ for some integer $a$. Then $d\neq 1$, since $j<kn+1$. Also note here that $1\leq a<d$, $\gcd(a,d)=1$, and $r$ is the smallest integer $s\geq 1$ such that $(k^s-1)a\equiv 0\bmod d$. Hence $r=\ord{k}{d}$, for $a$ and $d$ are coprime. The above argument can be reversed, so the desired result holds. 
\end{proof}

\begin{example}
\label{ex: k is 2 and n is 7}
 Let $k=2$ and $n=7$.
 The lengths of the cycles in the decomposition of the permutation $\sigma_{k,n}$ as a product of disjoint cycles belong to $\{ \ord{k}{d} \mid \text{$d$ divides $(kn+1)$ and $d\neq 1$} \}=\{2,4\}$ due to~\cref{cor: info sur les cycles}.
 They are $(1 \; 2 \; 4 \; 8)$, $(3 \; 6 \; 12 \; 9$), $(5 \; 10)$, and $(7 \;14 \; 13 \; 11)$.
 The divisors of $kn+1=15$ are $1,3,5,15$.
 According to~\cref{thm: precise description of cycles}, we show in~\cref{tab:Shallit's construction of cycles} all possible values of $d$ and $a$ to build the cycles of $\sigma_{k,n}$.
 Observe that the same cycle is rotated to obtain another cycle.
 Seed sets are, e.g., $\{1,3,5,7\}$ and $\{2,3,10,7\}$.
\begin{table}
    \centering
    \[
    \begin{array}{|c|c|c|}
\hline
d & a & \text{cycles}  \\ 
\hline
\multirow{2}{*}{3} & 1 & (5 \; 10) \\
& 2 & (10 \; 5) \\ 
\hline
\multirow{4}{*}{5} & 1 & (3 \; 6 \; 12 \; 9)\\ 
& 2 & (6 \; 12 \; 9 \; 3)\\
& 3 & (9 \; 3 \; 6 \; 12)\\
& 4 & (12 \; 9 \; 3 \; 6)\\
\hline
\end{array}
\quad
\begin{array}{|c|c|c|}
\hline
d & a & \text{cycles}  \\ 
\hline
\multirow{8}{*}{15} & 1 & (1 \; 2 \; 4 \; 8) \\
       & 2 & (2 \; 4 \; 8 \; 1) \\
       & 4 & (4 \; 8 \; 1 \; 2)\\
       & 7 & (7 \; 14 \; 13 \; 11)\\
       & 8 & (8 \; 1 \; 2 \; 4)\\
       & 11 & (11 \; 7 \; 14 \; 13)\\
       & 13 & (13 \; 11 \; 7 \; 14)\\
       & 14 & (14 \; 13 \; 11 \; 7)\\
\hline
\end{array}
    \]
    \caption{For $k=2$ and $n=7$, we use~\cref{thm: precise description of cycles} to obtain the cycles of the permutation $\sigma_{k,n}$.}
    \label{tab:Shallit's construction of cycles}
\end{table}
\end{example}

The \emph{order} of a permutation $\sigma$ is the smallest number of times the permutation must be applied to return to the identity permutation, i.e., the smallest integer $\ell$ such that $\sigma^\ell = \text{id}$.
In fact, the order of $\sigma$ is equal to the least common multiple of the lengths of cycles in its decomposition into a product of disjoint cycles.

\begin{proposition}
For all integers $k\ge 2$ and $n\ge 1$, the order of the permutation $\sigma_{k,n}$ is $\ord{k}{kn+1}$. 
\end{proposition}
\begin{proof}
Using~\cref{cor: info sur les cycles}, we obtain that the order of the permutation $\sigma_{k,n}$ is given by 
\begin{equation}\label{eq9}
L=\lcm\{ \ord{k}{d} \mid \text{$d$ is a divisor of $kn+1$ and $d\neq 1$}\}.
\end{equation}
To prove the statement, we show that $L= \ord{k}{kn+1}$.
For the sake of conciseness, let us write $\theta(d)=\ord{k}{d}$ for a divisor $d$ of $kn+1$.
By minimality of $\theta(d)$, we have that $\theta(d)$ divides $\theta(kn+1)$ when $d$ divides $kn+1$ (indeed, in this case, since $k^{\theta(kn+1)}- 1$ is a multiple of $kn+1$, then $k^{\theta(kn+1)}- 1$ is also a multiple of $d$), hence $L=\theta(kn+1)$ by the formula~\cref{eq9}.
\end{proof}

\begin{example}
Resuming~\cref{ex: k is 2 and n is 7}, we find that the order of the permutation $\sigma_{k,n}$ is given by $\lcm\{ 2,4 \} = 4$.
\end{example}

Given a permutation $\sigma$ on $[0,n]$, an \emph{inversion} is a pair $\{i,j\}$ of elements in $[0,n]$ such that $i<j$ implies $\sigma(i)>\sigma(j)$.
The permutation $\sigma$ is \emph{even} (resp., \emph{odd}) if its number of inversions is even (resp., odd).
The \emph{signature} $\sign(\sigma)$ of $\sigma$ is equal to $1$ (resp., $-1$) if $\sigma$ is even (resp., odd).
Looking again at~\cref{tab: first terms of various sequences C_k(n)}, we may compute the signature of the first few permutations $\sigma_{k,n}$: for $k\in\{2,3,6,7\}$, we obtain the sequence of signatures $-1,-1,1,1$ and for $k\in\{4,5\}$, the signature is always $1$.
We propose two ways to compute the signature of $\sigma_{k,n}$, one direct and the other using~\cref{cor: info sur les cycles}.

\begin{proposition}
\label{pro: number of inversions}
For all integers $k\ge 2$ and $n\ge 1$, the number of inversions of the permutation $\sigma_{k,n}$ is given by $\frac{(k-1)k}{2} \cdot \frac{n(n+1)}{2}$.
In particular, for each $k\ge 2$, the sequence $(\sign(\sigma_{k,n}))_{n\ge 1}$ of signatures of the permutation $\sigma_{k,n}$ is periodic with periods given by
\[
\begin{cases}
(-1,-1,1,1), & \text{if $k\equiv 2 \bmod{4}$}; \\
(-1,-1,1,1), & \text{if $k\equiv 3 \bmod{4}$}; \\
(1), & \text{if $k\equiv 0 \bmod{4}$}; \\
(1), & \text{if $k\equiv 1 \bmod{4}$}.
\end{cases}
\]
\end{proposition}
\begin{proof}
For each $\ell\in[1,k]$, define $I_\ell$ to be the interval of integers $[(\ell-1) n + 1, \ell n]$.
We note that we have $\sigma_{k,n}(x) = kx - (\ell-1) (kn+1)$ for $x\in I_\ell$ and $\ell\in[1,k]$.
(For example, $\sigma_{3,n}$ maps $x$ to $3x$ if $x\in[1,n]$, to $3x-(3n+1)$ if $x\in[n+1,2n]$, and to $3x-2(3n+1)$ if $x\in[2n+1,3n]$.)
The pair $\{i,j\}$ cannot be an inversion unless $i\in I_\ell$ and $j\in I_{\ell'}$ with $\ell,\ell'\in[1,k]$ and $\ell<\ell'$.
Couting these pairs of $\ell,\ell'$ gives
\[
\sum_{\ell=1}^{k-1} (k-\ell) = \sum_{\ell=1}^{k-1} \ell = \frac{(k-1)k}{2}.
\]
After analyzing the behavior of the map $\sigma_{k,n}$, the number of inversions $\{i,j\}$ with $i\in I_\ell$, $j\in I_{\ell'}$, $\ell,\ell'\in[1,k]$, and $\ell<\ell'$ is
\[
\sum_{m=1}^n m = \frac{n(n+1)}{2}.
\]
Indeed if we write $i=(\ell-1)n+m$ and $j=(\ell'-1)n+m'$ with $1\leq m,m'\leq n$, then $\sigma_{k,n}(i)>\sigma_{k,n}(j)$ if and only if $k(m'-m)<\ell'-\ell<k$, which means $1\leq m'\leq m$.
Putting together these two quantities gives the first part of the statement.

To get the second part of the statement, it is enough to carefully analyze, for all integers $k\ge 2$ and $n\ge 1$, the parity of $\frac{(k-1)k}{2} \cdot \frac{n(n+1)}{2}$ .
\end{proof}

\begin{corollary}
For all integers $k\ge 2$ and $n\ge 1$, the signature of the permutation $\sigma_{k,n}$ is determined by the parity of $kn - C_k(n)$.
\end{corollary}
\begin{proof}
Recall that a length-$\ell$ cycle is an even (resp., odd) permutation if $\ell$ is odd (resp., even).
Thus, using~\cref{cor: info sur les cycles}, we know that the signature of $\sigma_{k,n}$ is determined by the parity of
\[
\sum_{\substack{d | kn+1 \\ d \neq 1}}
\frac{\varphi(d)}{\ord{k}{d}} (\ord{k}{d}-1)
=  \sum_{\substack{d | kn+1 \\ d \neq 1}}
\varphi(d)- \sum_{\substack{d | kn+1 \\ d \neq 1}}
\frac{\varphi(d)}{\ord{k}{d}} = kn - C_k(n),
\]
where, for the last equality, we used~\cref{lem: famous formula with varphi,pro: number of cycles in sigma kn}.
\end{proof}

We note that combining the previous two results allows us to determine the parity of the number $C_k(n)$ for all integers $k\ge 2$ and $n\ge 1$.

Given an integer $n\ge 1$, let us consider the permutation $\mu_n$ over $[1,n]$ that maps $x$ to $\frac{x}{2}$ if $x$ is even, and $n-\frac{x-1}{2}$ otherwise, and referred to as the \emph{Queneau-Daniel permutation}~\cite{Bringer-1969}.
An integer $n$ is a \emph{Queneau number} if the corresponding permutation $\mu_n$ is a cycle of length $n$.
The first few Queneau numbers are $1, 2, 3, 5, 6, 9, 11, 14, 18$ (also see~\cite[A054639]{Sloane}).
They were first analyzed by Queneau, then Bringer 
proposed a first systematic mathematical study of these 
numbers and their generalizations~\cite{Bringer-1969}; then
they were fully characterized by Dumas~\cite{Dumas-2008}, 
also see~\cite{Vallet2010}, the history 
in~\cite{Saclolo-2011}, and the survey~\cite{Asveld-2013}.
In the case of the permutation $\sigma_{k,n}$ one could ask a similar question, i.e., for which integers $n$ do we have $C_k(n)=1$?
We call these integers \emph{Queneau-like numbers with respect to $k$}.
In view of~\cref{tab: first terms of various sequences C_k(n)}, we are able to compute the first few Queneau-like numbers with respect to the first few values of $k\in[2,7]$ in~\cref{tab:Queneau-like numbers}.
Queneau-like numbers with respect to $k=2$ are related to the sequence~\cite[A163782]{Sloane} (see~\cref{cor: J2 prime iff Queneau-like for 2} below); the other sequences of Queneau-like numbers do not seem to appear in the OEIS~\cite{Sloane}.
We obtain the following characterization of Queneau-like numbers.
\begin{table}
    \centering
    \[
    \begin{array}{c|c|c|c|c|c|c}
        k & 2 & 3 & 4 & 5 & 6 & 7\\
        \hline
        \text{Queneau-like numbers} & 1, 2, 5, 6, 9, 14, 18 & 2, 6, 10, 14 & \emptyset & \emptyset & 2, 10, 13, 17, 18 & 10, 18 \\
        \text{w.r.t $k$ in $[1,20]$} &&&&&& 
    \end{array}
    \]
    \caption{For $k\in [2,7]$, the first few Queneau-like numbers with respect to $k$.}
    \label{tab:Queneau-like numbers}
\end{table}

\begin{proposition}
\label{pro: Queneau-like numbers}
Let $k\ge 2$ and $n\ge 1$ be integers.
The following are equivalent.
\begin{enumerate}
    \item[{(1)}] The integer $n$ is a Queneau-like number with respect to $k$.
    \item[{(2)}] The permutation $\sigma_{k,n}$ is a cycle of length $kn$.
    \item[{(3)}] The integer $kn+1$ is a prime number and $k$ is a primitive root modulo $(kn+1)$.
    \item[{(4)}] The multiplicative order of $k$ modulo $(kn+1)$ is $kn$, i.e., $\ord{k}{kn+1}=kn$.
\end{enumerate}
\end{proposition}
\begin{proof}
The equivalence $(1) \Leftrightarrow (2)$ follows from the definition of Queneau-like numbers.

We show $(2) \Leftrightarrow (3)$.
We first note that the permutation $\sigma_{k,n}$ is a cycle of length $kn$ if and only if $C_k(n)=1$.
Using~\cref{pro: number of cycles in sigma kn}, we obtain that this condition is equivalent to ask that
\[
\sum_{\substack{d | kn+1 \\ d \neq 1}}
\frac{\varphi(d)}{\ord{k}{d}} = 1.
\]
As each term in the sum is a positive number, $C_k(n)=1$ if and only if $kn+1$ is a prime and $\frac{\varphi(kn+1)}{\ord{k}{kn+1}}=1$.
As the multiplicative group $(\Z/(kn+1)\Z)^*$ is of cardinality $\varphi(kn+1)$, we obtain that the last part of the condition is equivalent to asking that $k$ is a primitive root modulo $kn+1$, as desired.

We finally show that $(2) \Leftrightarrow (4)$.
Due to the previous paragraph, Item~$(2)$ rewrites: $kn+1$ is a prime and $\ord{k}{kn+1}=kn$, by definition of the Euler totient function. 
Now we show that the first part of the latter condition is superfluous.
Suppose that $\ord{k}{kn+1}=kn$. We always have that
$\gcd(k, kn+1) = 1$, thus $k$ and its powers are invertible
modulo $kn+1$. If $\ord{k}{kn+1} = kn$, this means that
the multiplicative group generated by $k$ has cardinality
$kn$, thus is equal to the set of all non-zero elements of 
$\Z/(kn+1)\Z$. In particular, all non-zero elements of 
$\Z/(kn+1)\Z$ are invertible, so $kn+1$ is a prime number, 
as desired.
\end{proof}

\begin{remark}
\label{rk: Josephus problem}
An alternative proof of the equivalence between Items $(2)$ and $(4)$ of~\cref{pro: Queneau-like numbers} is to use~\cref{thm: precise description of cycles}.
Indeed, we notably want to find a divisor $d$ of $kn+1$ with $d\neq 1$ such that $kn=\ord{k}{d}$.
If $d\neq kn+1$, then $d<kn$ and the condition cannot be fulfilled; so $d=kn+1$, as expected.

We note that Item~$(3)$ reminds us of the result from~\cite[Corollaires page~8]{Allouche-1983-1984} for the case $k=2$.

While we are able to eliminate the condition that 
$kn+1$ is a prime number when assuming Item~$(4)$ 
of~\cref{pro: Queneau-like numbers}, we note that it may
happen that $k$ is a primitive root modulo $kn+1$ but 
$kn+1$ is not a prime number. For example, for $k=2$ and 
$n=4$, we have that $2n+1=9$ is composite but $2$ generates 
$(\Z/9\Z)^{\times}$.

We also note the following question: let us fix some 
integer $k\ge 2$; then, is the set of Queneau-like numbers 
with respect to $k$ infinite?
Due to Item~$(4)$, we ask whether $k$ is a primitive root 
modulo infinitely many primes of the form $kn+1$.
This is a difficult question to answer since it is related 
to Artin's primitive root conjecture of 1927 (see the survey~\cite{Moree-2012}).

We end this remark by mentioning {\em conjugate permutations}.
Recall that two permutations $\sigma$ and $\tau$ are \emph{conjugates} if there exists a permutation $\pi$ such that $\tau = \pi \sigma \pi^{-1}$, where we let $\pi^{-1}$ denote the \emph{inverse (permutation)} of $\pi$.
It is clear that conjugate permutations share the same number of cycles in their decomposition into disjoint cycles and their cycles have the same lengths, although the actual elements in the cycles may differ (actually the converse is also true).
In our case, the Queneau-Daniel permutation $\mu_{2n}$ and the $(2,n)$-perfect shuffle permutation $\sigma_{2,n}$ are not necessarily conjugate, e.g., see~\cref{tab:comparaison between mu and sigma}.
\begin{table}
    \centering
    \[
    \begin{array}{c|c|c|c|c|c}
        n & 1 & 2 & 3 & 4 & 5 \\
        \hline
        \mu_{2n} & (1 \; 2) & (1 \; 4 \; 2), (3) & (1 \; 6 \; 3 \; 5 \; 4 \; 2) & (1 \; 8 \; 4 \; 2), (3 \; 7 \; 5 \; 6) &  (1 \; 10 \; 5 \; 8 \; 4 \; 2), (3 \; 9 \; 6), (7)\\
        \hline
        \sigma_{2,n} & (1 \; 2) & (1 \; 2 \; 4 \; 3) & (1 \; 2 \; 4), (3 \; 6 \; 5) & (1 \; 2 \; 4 \; 8 \; 7 \; 5), (3 \; 6)  & (1 \; 2 \; 4 \; 8 \; 5 \; 10 \; 9 \; 7 \; 3 \; 6) 
    \end{array}
    \]
    \caption{For $n\in [1,5]$, we compare the cycle decomposition of the Queneau-Daniel permutation $\mu_{2n}$ and our $(2,n)$-perfect shuffle permutation $\sigma_{2,n}$.}
    \label{tab:comparaison between mu and sigma}
\end{table}
\end{remark}

It turns out that Queneau-like numbers with respect to $k=2$ are related to the famous Josephus problem (e.g., see~\cite[Section~1.3]{Graham-Knuth-Patashnik-1994}).
Let $n\ge 2$ be an integer and put the numbers from the interval $[1,n]$ on a circle.
In a cyclic way, mark the second unmarked number until all $n$ numbers are marked.
For example, with $n=6$, we mark $2$, then $4$, $6$, $3$, $1$, and finally $5$.
Considering the permutation defined by the order in which the numbers are marked, we say that $n$ is a \emph{$J_2$-prime number} if this permutation consists of a single cycle of length $n$.
For every $k\ge 2$, \emph{$J_k$-prime numbers} are defined analogously by marking every $k$th number instead.
The sequences of $J_k$-primes for $k\in [2,20]$ are indexed by~\cite[A163782-A163800]{Sloane}.
We next show that $J_2$-primes are Queneau-like numbers with respect to $k=2$, and vice versa.
For larger values $k$, we note that our Queneau-like numbers with respect to $k$ differ from $J_k$-primes.

\begin{corollary}
\label{cor: J2 prime iff Queneau-like for 2}
Let $n\ge 2$ be integer.
Then $n$ is a $J_2$-prime number if and only if $n$ is a Queneau-like number with respect to $k=2$.    
\end{corollary}
\begin{proof}
By the equivalence $(1)\Leftrightarrow (3)$ of~\cref{pro: Queneau-like numbers}, it suffices to show that 
an integer $n\ge 2$ is a $J_2$-prime number if and only if $2n+1$ is a prime number and $2$ generates $(\Z/(2n+1)\Z)^{\times}$.
Recall that by~\cite[Theorem~5.12]{Asveld-2011}, an integer $n\ge 2$ is a $J_2$-prime number if and only if $2n+1$ is a prime number and exactly one of the following two conditions holds:
\begin{enumerate}
    \item[{(1)}] $n\equiv 1\bmod{4}$, and $2$ generates $(\Z/(2n+1)\Z)^{\times}$, but $-2$ does not;
    \item[{(2)}] $n\equiv 2\bmod{4}$, and 
 both $2$ and $-2$ generate $(\Z/(2n+1)\Z)^{\times}$.
\end{enumerate}
 Hence if $n$ is a $J_2$-prime number, then  $2n+1$ is a prime number and $2$ generates $(\Z/(2n+1)\Z)^{\times}$, as desired.
 
 Conversely, assume that $2n+1$ is prime and $2$ generates $(\Z/(2n+1)\Z)^{\times}$.
 We distinguish two cases below to show that $n$ satisfies the sufficient condition of~\cite[Theorem~5.12]{Asveld-2011}, so that $n$ is a $J_2$-prime number.
 
 {\bf Case 1:} First assume that $n\equiv 1,2 \bmod{4}$. Since 
$\ord{2}{2n+1}=2n$ and  $(2^{n})^2\equiv 1\bmod (2n+1)$, we have $2^n\equiv -1\bmod (2n+1)$, for $\Z/(2n+1)\Z$ is a field. If $n\equiv 1\bmod{4}$, then $n$ is odd, and thus $(-2)^n\equiv 1\bmod (2n+1)$, so $-2$ cannot generate 
$(\Z/(2n+1)\Z)^{\times}$.  If $n\equiv 2\bmod{4}$, then $n$ is even, and we have $(-2)^n=2^n\equiv -1\bmod (2n+1)$, so $\ord{-2}{2n+1}=2n$, and $-2$ generates 
$(\Z/(2n+1)\Z)^{\times}$. 

{\bf Case 2:} Now assume that $n\equiv 0,3 \bmod{4}$. 
Then we can write $2n+1=8m\pm 1$ for some integer $m$. By~\cite[Proposition~3.9]{Asveld-2011}, there exists an integer $x\in \Z/(2n+1)\Z$ such that $x^2 \equiv 2 \bmod{(2n+1)}$.
Since $2n+1$ is prime, the cardinality of 
$(\Z/(2n+1)\Z)^{\times}$ is $2n$, so we have 
$2^n \equiv x^{2n} \equiv 1 \bmod{(2n+1)}$.
This implies  $2n=\ord{2}{2n+1}\le n$, which 
is absurd and Case 2 cannot occur.
\end{proof}

Using the reasoning in the proof of~\cref{pro: number of cycles in sigma kn} or~\cref{cor: info sur les cycles}, we obtain the following result.
In particular, $C_k(n)$ gives the multiplicity of $1$ as a zero of some polynomial (also see~\cite[Corollaires page~8]{Allouche-1983-1984} for the case $k=2$).

\begin{corollary}
\label{cor: char poly f}
Let $k\ge 2$ and $n\ge 1$ be integers and let $(e_1, e_2, \dots, e_{kn})$ denote the canonical basis of $\R^{kn}$.
Let $f_{k,n}$ be the endomorphism of $\R^{kn}$ defined by $f_{k,n}(e_j) := e_{\sigma_{k,n}(j)}$ for every integer $j\in[1,kn]$, where $\sigma_{k,n}$ is the $(n,k)$-perfect shuffle permutation defined above.
Then the characteristic polynomial of $f_{k,n}$ is
\[
\chi_{f_{k,n}}(X)= \prod_{\substack{d | kn+1 \\ d \neq 1}}
(X^{\ord{k}{d}} - 1)^{\frac{\varphi(d)}{\ord{k}{d}}}.
\] 
\end{corollary}
\begin{proof}
Let $M_{k,n}$ denote the matrix associated with $f_{k,n}$ in the canonical basis of $\R^{kn}$.
From a classical result on the characteristic polynomial of a permutation matrix (e.g., see~\cite[Section~1]{Aitken-1936}), we obtain
\begin{align}
\label{eq: char poly of permutation matrix}
 \chi_{f_{k,n}}(X) = \det(X I - M_{k,n}) = \prod_{\ell} (X^\ell - 1)^{c_\ell},
\end{align}
where $c_\ell$ is the number of length-$\ell$ cycles in the decomposition of $\sigma_{k,n}$ in a product of cycles. 
(To prove~\cref{eq: char poly of permutation matrix}, it is enough to reorder the elements $e_1,\ldots,e_{kn}$ following the cycles of the decomposition of $\sigma_{k,n}$: first, we start with the vectors $e_1, e_{\sigma_{k,n}(1)}, e_{\sigma_{k,n}^2(1)}, \ldots, e_{\sigma_{k,n}^{t-1}(1)}$, where $t$ is the length of the cycle containing $1$, then we choose a non-already-visited vector $e_m$, and so on and so forth.)

Using the reasoning in the proof of~\cref{pro: number of cycles in sigma kn} leading to computing the possible lengths of cycles and numbers of fixed length cycles in $\sigma_{k,n}$ or~\cref{cor: info sur les cycles}, Equality~\cref{eq: char poly of permutation matrix} becomes
\[
\chi_{f_{k,n}}(X) 
= \prod_{\substack{d | kn+1 \\ d \neq kn+1}}
\left(X^{\ord{k}{\frac{kn+1}{d}}} - 1\right)^{\varphi(\frac{kn+1}{d})/ \ord{k}{\frac{kn+1}{d}}} 
= \prod_{\substack{d | kn+1 \\ d \neq 1}}
\left(X^{\ord{k}{d}} - 1\right)^{\frac{\varphi(d)}{\ord{k}{d}}},
\]
where the second equality is obtained after the change of index
$d' = (kn+1)/d$.
\end{proof}

\begin{example}
We resume~\cref{ex: k=3 and n=5} for which $k=3$ and $n=5$.
If we look at Equality~\cref{eq: char poly of permutation matrix}, we have $1$ length-$1$, $3$ length-$2$ and $2$ length-$4$ cycles in $\sigma_{k,n}$, so in the case the characteristic polynomial of $f_{k,n}$ is $(X-1)(X^2-1)^3(X^4-1)^2$.
This coincides with the formula in~\cref{cor: char poly f} since we obtain 
\[
\left\{\left( \ord{k}{d}, \dfrac{\varphi(d)}{\ord{k}{d}} \right) \Big\lvert \, d \text{ divides } kn+1 \text{ and } d \neq 1\right\}
= \{(1,1),(2,1),(2,2),(4,2)\}.
\]
\end{example}

\section{A second interpretation through arithmetic and algebra}
\label{sec: second interpretation}

For all integers $k\ge 2$ and $n\ge 1$, we define 
the quantity
\begin{equation}\label{eq:ik}
i_k(n):=\sum_{d|n} \frac{\varphi(d)}{\ord{k}{d}}\cdot
\end{equation}
Note that our definition is the same as the definition given
in~\cite{Deaconescu-2008,Moree-2012,Moree-Sole-2005,Pomerance-Shparlinski-2010}, 
but slightly different from the one given 
in~\cite{Rogers-1996,Vasiga-Shallit-2004}.
For $k\in [2,7]$, the first few terms of the 
sequence $(i_k(kn+1))_{n\ge 1}$ are given 
in~\cref{tab: first terms of various sequences i_k(n)}.
The sequence corresponding to $k=2$ is~\cite[A081844]{Sloane}  
(also see the related sequence~\cite[A000374]{Sloane}), while the others do not seem to be indexed in Sloane's OEIS~\cite{Sloane}.
Also note that $i_k(kn+1)$ and $C_k(n)$ only differ by $1$.
\begin{table}[ht]
    \centering
    \[
\begin{array}{c|l}
    k & (i_k(kn+1))_{n\ge 1} \\
    \hline
    2 & 2, 2, 3, 3, 2, 2, 5, 3, 2, 6, 3, 3, 4, 2, 7, 5, 6, 2, 5, 3, \ldots \\
    3 & 3, 2, 4, 5, 7, 2, 6, 3, 7, 2, 4, 3, 13, 2, 6, 3, 15, 6, 4, 7, \ldots \\
    4 & 3, 5, 3, 5, 9, 5, 3, 9, 3, 5, 15, 5, 3, 9, 3, 13, 9, 9, 9, 9, \ldots \\
    5 & 4, 3, 8, 5, 8, 11, 12, 3, 4, 5, 14, 3, 12, 15, 12, 5, 4, 11, 26, 5, \ldots \\
    6 & 4, 2, 3, 9, 6, 10, 15, 7, 10, 2, 3, 3, 2, 10, 11, 9, 2, 2, 15, 3, \ldots \\
    7 & 5, 6, 4, 5, 15, 8, 14, 21, 17, 2, 12, 7, 7, 10, 6, 9, 39, 2, 4, 9, \ldots \\
\end{array}
\]
    \caption{For $k\in [2,7]$, the first few terms of the sequence $(i_k(kn+1))_{n\ge 1}$.}
    \label{tab: first terms of various sequences i_k(n)}
\end{table}
\begin{remark}\label{rk:ik-irreducible-factors}
If $q$ is the order of a finite field, then, applying~\cite[Lemma 5]{Moree-Sole-2005} to the particular case where we look at divisors of $qn+1$ shows that $i_q(qn+1)$ gives the number of distinct irreducible factors of the polynomial $X^{qn+1} -1$ in $\mathbb{F}_q[X]$.
For example, the sequence~\cite[A081844]{Sloane} corresponds to $q=2$.
\end{remark}

Another way to write $i_k(kn+1)$ (or $C_k(n)$) is the 
following one, to which we propose two independent 
proofs: one with an arithmetical flavor, the other on 
the algebraic side.

\begin{proposition}
\label{pro: identity with phi and ord}
For all integers $k\ge 2$ and $n\ge 1$, we have     
\[
i_k(kn+1) =
\frac{1}{\ord{k}{kn+1}} 
\sum_{j = 0}^{\ord{k}{kn+1} - 1} \gcd(k^j - 1, kn + 1).
\]
\end{proposition}  

\begin{remark}\label{rem:without-zero}
Here we have used the usual convention that $\gcd(0,r) = r$
for every integer $r\geq 1$.
Note that, for all integers $k\ge 2$ and $n\ge 1$, we also have
\begin{align}
\label{eq: equality with gcd with and without 0}
    \sum_{j = 0}^{\ord{k}{kn+1} - 1} \gcd(k^j - 1, kn + 1)
= \sum_{j = 1}^{\ord{k}{kn+1}} \gcd(k^j - 1, kn + 1).
\end{align}
\end{remark}

\begin{proof}[``Arithmetical'' proof of~\cref{pro: identity with phi and ord}] For every integer $n\ge 1$, we let $U(n)$ be the quantity
\[
U(n) := \sum_{x=1}^{kn} A(x), 
\]
where $A(x) := \Card \{j \in [0, \ord{k}{kn+1} - 1] \mid (k^j-1)x  \equiv 0 \bmod{(kn+1)}\}$.

Fix some $j \in [0, \ord{k}{kn+1} - 1]$ and take $x\in[1,kn]$.
We have
\[
(k^j-1)x  \equiv 0 \bmod{(kn+1)}
\Leftrightarrow
(k^j-1)x'  \equiv 0 \bmod{\left(\frac{kn+1}{\gcd(x, kn+1)}\right)},
\]
where $x' := \frac{x}{\gcd(x, kn+1)}$.
Since $\gcd\left(x', \frac{(kn+1)}{\gcd(x, kn+1)}\right) = 1$, we get
\[
(k^j-1)x'  \equiv 0  \bmod{\left(\frac{kn+1}{\gcd(x, kn+1)}\right)}
\Leftrightarrow
(k^j-1) \equiv 0 \bmod{\left(\frac{kn+1}{\gcd(x, kn+1)}\right)}.
\]
Thus, we find
\begin{align*}
A(x)
&= \Card \left\{j \in [0, \ord{k}{kn+1} - 1] \mid j \text{ is a multiple of } \ord{k}{\frac{(kn+1)}{\gcd(x, kn+1)}}\right\} 
\\
&= \frac{\ord{k}{kn+1}}{\ord{k}{\frac{kn+1}{\gcd(x, kn+1)}}}.
\end{align*}
(We may even verify that $\ord{k}{\frac{kn+1}{\gcd(x, kn+1)}}$ divides $\ord{k}{kn+1}$.)
Observing that $x\in [1,kn]$ implies that $\gcd(x,kn+1) < kn+1$, we obtain that
\begin{align*}
\frac{U(n)}{\ord{k}{kn+1}} 
&= 
 \displaystyle\sum_{x=1}^{kn} \frac{1}{\ord{k}{\frac{kn+1}{\gcd(x, kn+1)}}} \\
&= \sum_{\stackrel{\scriptstyle d | kn+1}{d \neq kn+1}} 
\   \sum_{\stackrel{\scriptstyle 1 \leq x \leq kn}{\gcd(x,kn+1)=d}}  
\frac{1}{\ord{k}{\frac{kn+1}{d}}} \\
&= \displaystyle\sum_{\stackrel{\scriptstyle d | kn+1}{d \neq kn+1}} 
\frac{1}{\ord{k}{\frac{kn+1}{d}}}
\sum_{\stackrel{\scriptstyle 1 \leq x < kn+1}{\gcd(x,kn+1)=d}} 1.
\end{align*}
We apply the change of index $x'=x/d$ to obtain
\[
\frac{U(n)}{\ord{k}{kn+1}} 
= 
 \displaystyle\sum_{\stackrel{\scriptstyle d | kn+1}{d \neq kn+1}} 
\frac{1}{\ord{k}{\frac{kn+1}{d}}}
\sum_{\stackrel{\scriptstyle 1 \leq x' < (kn+1)/d}{\gcd(x', (kn+1)/d)=1}} 1
= \displaystyle\sum_{\stackrel{\scriptstyle d | kn+1}{d \neq kn+1}}
 \frac{\varphi\left(\frac{kn+1}{d}\right)}{\ord{k}{\frac{kn+1}{d}}}
\]
by definition of $\varphi$.
Using the change of index $d'=(kn+1)/d$, we obtain
\begin{align}
\label{eq: relation U(n) 1}
\frac{U(n)}{\ord{k}{kn+1}} 
= 
 \sum_{\stackrel{\scriptstyle d | kn+1}{d \neq 1}} \frac{\varphi(d)}{\ord{k}{d}}.
\end{align}

On the other hand, we can write 
\begin{align*}
   U(n) 
&= \Card  \{(j,x) \in [0, \ord{k}{kn+1} - 1] \times [1, kn] \mid (k^j-1)x  \equiv 0 \bmod{(kn+1)}\}
\\
&=
\sum_{j = 0}^{\ord{k}{kn+1} - 1} \Card  \{x \in [1, kn] \mid (k^j-1)x \equiv 0 \bmod{(kn+1)}\},
\end{align*}
so that we have
\begin{equation}
\label{eq: relation U(n) 2}
U(n) = 
\sum_{j = 0}^{\ord{k}{kn+1} - 1} \Card  \{x \in [1, kn] \mid (k^j-1)x \equiv 0 \bmod{(kn+1)}\}.
\end{equation}
For $j\in [0, \ord{k}{kn+1} - 1]$, let $d := \gcd(k^j - 1, kn+1)$.
Since $(k^j - 1)/d$ is coprime to $(kn+1)/d$, we have
\[
(k^j - 1)x \equiv 0 \bmod{(kn+1)}
\Leftrightarrow 
\frac{k^j - 1}{d} x \equiv 0 \bmod{\frac{kn+1}{d}} 
\Leftrightarrow 
x \equiv 0 \bmod{\frac{kn+1}{d}}.
\]
This condition together with the fact that $1\le x < kn+1$ is equivalent to saying that there exists an integer $\lambda \in [1, d) = [1, d-1]$ such that $x = \lambda \frac{kn+1}{d}$. Therefore, 
\begin{align*}
\Card  \{x \in [1, kn] \mid (k^j-1)x \equiv 0 \bmod{(kn+1)}\}
&=
\Card \left\{\lambda \in [1, d-1] \mid \lambda \frac{kn+1}{d} \in [1, kn+1)\right\} \\
&= d - 1 \\
&= \gcd(k^j - 1, kn+1) - 1,
\end{align*}
so using~\cref{eq: relation U(n) 2}, we find
\begin{align}
\label{eq: relation U(n) 3}    
\frac{U(n)}{\ord{k}{kn+1}} = 
\frac{1}{\ord{k}{kn+1}} 
\left(\sum_{j = 0}^{\ord{k}{kn+1} - 1} \gcd(k^j - 1, kn+1)\right) - 1.
\end{align}
Comparing~\cref{eq: relation U(n) 1,eq: relation U(n) 3} gives
\begin{align*}
\sum_{\stackrel{\scriptstyle d | kn+1}{d \neq 1}} \frac{\varphi(d)}{\ord{k}{d}}
=
 \frac{1}{\ord{k}{kn+1}} 
\left(\sum_{j = 0}^{\ord{k}{kn+1} - 1} \gcd(k^j - 1, kn+1)\right) - 1,
\end{align*}
which, after adding $1$ on each side and re-absorbing it in the sum in the left-hand side as the term $\frac{\varphi(1)}{\ord{k}{1}}$ for $d=1$ equals $1$, yields the desired result.
\end{proof}

\begin{proof}[Algebraic proof of~\cref{pro: identity with phi and ord}]
For the sake of readability, let $m=kn+1$.
The equality we need to prove is equivalent to
\begin{equation}
\label{eq: plus1}
\sum_{d | m} \varphi(d) \frac{\ord{k}{m}}{\ord{k}{d}} 
=
\sum_{j = 0}^{\ord{k}{m} - 1} \gcd(k^j - 1, m).
\end{equation}

Since $k$ and $m$ are coprime, $k$ belongs to the group $(\Z/m\Z)^{\times}$ of invertible elements of $\Z/m\Z$.
Let $\langle k \rangle_m$ be the subgroup of $(\Z/m\Z)^{\times}$ generated by $k$: it is a group of order $\ord{k}{m}$. 

Let $d$ be a divisor of $m$.
There exists a unique ring morphism from $\Z/m\Z$ to $\Z/d\Z$, which is surjective and its kernel is the ideal $d(\Z/m\Z)$.
The induced morphism from $(\Z/m\Z)^{\times}$ to $(\Z/d\Z)^{\times}$ maps $\langle k \rangle_m$ to $\langle k \rangle_d$.
Consequently, $\langle k \rangle_d$ is the quotient of $\langle k \rangle_m$ by the sub-group $\{y \in \langle k \rangle_m : y \equiv 1 \bmod d\}$.
By Lagrange's theorem, the left-hand side of~\cref{eq: plus1} becomes
\begin{align*}
\sum_{d | m} \varphi(d) \frac{\ord{k}{m}}{\ord{k}{d}} 
&=
\sum_{d | m} \varphi(d) \frac{\Card \langle k \rangle_m}
{\Card \langle k \rangle_d}  \\
&=
\sum_{d | m} \varphi(d) \cdot \Card \{y \in \langle k \rangle_m \mid y \equiv 1 \bmod d\} \\
&= \sum_{d | m} \varphi(d) \cdot \Card \{j \in [0, \ord{k}{m} - 1] \mid \textrm{$d$ divides $(k^j - 1)$}\}.
\end{align*}
The formula in~\cref{eq: set cardinality} yields
\[
\sum_{d | m} \varphi(d) \frac{\ord{k}{m}}{\ord{k}{d}} 
= \sum_{j \in [0, \ord{k}{m} - 1]} \sum_{\stackrel{\scriptstyle d | m}{d | (k^j-1)}} \varphi(d)
= \sum_{j=0}^{\ord{k}{m} - 1} \sum_{d | \gcd(k^j - 1, m)} \varphi(d).
\]
Now the formula of~\cref{lem: famous formula with varphi} gives the right-hand side of~\cref{eq: plus1}, as desired.
\end{proof}

Applying~\cref{pro: identity with phi and ord} to the specific case $k=2$ gives the equality between~\cref{expression1,expression2}.

\begin{corollary}
\label{cor: equality between our two quantities}
For each odd integer $m\ge 3$, we have
\[ \sum_{\substack{d | m \\ d \neq 1}}
 \frac{\varphi(d)}{\ord{2}{d}} =
 \left(\frac{1}{\ord{2}{m}} 
 \sum_{j = 0}^{\ord{2}{m} - 1} \gcd(2^j - 1, m)\right) - 1.
 \]
\end{corollary}
\begin{proof}
From~\cref{pro: identity with phi and ord} applied to the case $k=2$, we obtain
\[
 \sum_{d | m} \frac{\varphi(d)}{\ord{2}{d}} = \frac{1}{\ord{2}{m}} 
 \sum_{j = 0}^{\ord{2}{m} - 1} \gcd(2^j - 1, m).
\]
As the term corresponding to $d=1$ in the left-hand side of the previous equality is equal to $\frac{\varphi(d)}{\ord{2}{d}} = 1$, we obtain the desired result.
\end{proof}

Actually a third quantity is equal to both members of 
the equality in~\cref{pro: identity with phi and ord}.
Before stating this equality in \cref{pro:troisieme-larron} below, we will give two lemmas.

\begin{lemma}[Apostol]\label{Apostol}
If $a, b, m, n$ are positive integers with $a$ and $b$ 
relatively prime and $a > b$, then $\gcd(a^n - b^n, a^m - b^m) = a^{\gcd(m,n)} - b^{\gcd(m,n)}$.
\end{lemma}

\begin{proof}
This result is well-known for the particular case $b=1$
(which is actually the one we will use). The general case
can be found as a problem posed in~\cite[page~49]{Apostol1}, 
with a solution given in~\cite[pages~86--87]{Apostol2}.
\end{proof}

\begin{lemma}\label{lem:surprising}
For all integers $k\ge 2$ and $n\ge 1$, we have the following equalities:    
\begin{align*}
\sum_{j = 0}^{\ord{k}{kn+1} - 1}  \gcd(k^j - 1, kn + 1) 
&= \sum_{j = 1}^{\ord{k}{kn+1}} \gcd(k^j - 1, kn + 1)\\
&= \sum_{d | \ord{k}{kn+1}} \varphi\left(\frac{\ord{k}{kn+1}}{d}\right)
\gcd(k^d - 1, kn + 1).
\end{align*}
\end{lemma}

\begin{proof}
The first equality was already indicated in~\cref{eq: equality with gcd with and without 0} of~\cref{rem:without-zero}.
We now turn to the proof of the second equality.
In order to simplify notation in the proof, we fix $k$ and $n$, and we write $\theta := \theta(k,n) := \ord{k}{kn+1}$ and, for the left-hand side of the equality that we want to prove,
\[
A := A(k,n) := 
\sum_{1 \leq j \leq \theta} \gcd(k^j - 1, kn + 1). 
\]
We group terms in $A$ according to the value of $\gcd(\theta,j)$.
Since the values of this gcd are divisors of $\theta$, we have
\begin{equation}\label{A-double-sum}
A = \sum_{d|\theta}
\ \ \sum_{\substack{1 \leq j \leq \theta \\ 
\gcd(\theta,j) = d}} \gcd(k^j - 1, kn + 1).
\end{equation}
Now, since we have $k^{\theta} \equiv 1 \bmod (kn+1)$, there exists a positive integer $x$ such that 
$k^{\theta} - 1 = x(kn+1)$. Thus, 
\begin{equation}\label{eq:intermediate}
x \gcd(k^j - 1, kn+1) = \gcd(x(k^j-1),x(kn+1))
= \gcd(x(k^j-1), k^{\theta} - 1).
\end{equation}
If we let $\gcd(\theta,j) = d$, we have that the right-hand side of~\cref{eq:intermediate} is equal to 
\begin{equation}\label{eq:intermediate-2}
\gcd(x(k^j-1), k^{\theta} - 1)  = (k^{d} - 1) 
\gcd\left(x\frac{k^j-1}{k^{d}-1}, \frac{k^\theta-1}{k^{d}-1}\right).
\end{equation}
By~\cref{Apostol}, we have $\gcd(k^j-1,k^{\theta}-1) = k^{d}-1$, so $\frac{k^j-1}{k^{d}-1}$ and
$\frac{k^\theta-1}{k^{d}-1}$ are coprime.
Thus, we have
\[
\gcd\left(x\frac{k^j-1}{k^{d}-1}, \frac{k^\theta-1}{k^{d}-1}\right) = \gcd\left(x, \frac{k^\theta-1}{k^{d}-1}\right)
\]
and the right-hand side of~\cref{eq:intermediate-2} is equal to 
\begin{align*}
\gcd(x(k^j-1), k^{\theta} - 1) 
&= \gcd(x(k^{d} - 1), k^{\theta} - 1) 
= \gcd(x(k^{d} - 1), x(kn+1)) \\
&= x \gcd(k^{d} - 1, kn+1).
\end{align*}
Comparing with~\cref{eq:intermediate} yields $\gcd(k^j - 1, kn+1) = \gcd(k^{d} - 1, kn+1)$.
Finally, ~\cref{A-double-sum} becomes
\[
\begin{array}{lll}
A &=& \displaystyle\sum_{d|\theta}
\ \ \sum_{\substack{1 \leq j \leq \theta \\ 
\gcd(\theta, j) = d}}
\gcd(k^{d} - 1, kn+1) \\
&=& \displaystyle\sum_{d|\theta} \gcd(k^{d} - 1, kn+1)
\sum_{\substack{1 \leq j \leq \theta \\ \gcd(\theta, j) = d}} 1 \\
&=& \displaystyle\sum_{d|\theta} 
\gcd(k^{d} - 1, kn+1)
\varphi\left(\frac{\theta}{d}\right),
\end{array}
\]
where, for the third equality, we have used that
\[
\sum_{\substack{1 \leq j \leq \theta \\ \gcd(\theta, j) = d}} 1 = \sum_{\substack{1 \leq j' \leq \theta/d \\ \gcd(\theta/d, j') = 1}} 1 \ = \
\varphi\left(\frac{\theta}{d}\right),
\]
which finishes the proof.
\end{proof}

As a corollary and using 
\cref{pro: identity with phi and ord}, 
we have the following result, which is the case
$a=k$ and $m = kn+1$ of~\cite[Theorem page~2]{Deaconescu-2008}.

\begin{corollary}\label{pro:troisieme-larron}
For all integers $k\ge 2$ and $n\ge 1$, we have  
\[
i_k(kn+1) 
= \frac{1}{\ord{k}{kn+1}}
\sum_{d | \ord{k}{kn+1}} \varphi\left( \frac{\ord{k}{kn+1}}{d}\right) \gcd(k^d - 1, kn+1).
\]
\end{corollary}

\begin{remark}
For an even more detailed study of various 
similar families of permutations, the reader 
can consult~\cite{Patil-Storch-2011}.
\end{remark}

\section{Asymptotics}
\label{sec: asymptotics}

Now we look at the asymptotics of the sequences $(C_k(n))_{n\ge 1}$ and $(i_k(kn+1))_{n\ge 1}$ (also see~\cite[Corollaires page~8]{Allouche-1983-1984} for the case $k=2$).
In the following, we let $\vlog$ denote the natural logarithm (in base $e$).
For two sequences $(U(n))_{n\ge 0}$ and $(V(n))_{n\ge 0}$, we also recall the notation $U(n) = O(V(n))$ if there exists a positive constant $c$ such that $|U(n)| \le c |V(n)|$ for every sufficiently large $n$.

\begin{proposition}
\label{pro: asymptotics}
For all integers $k\ge 2$ and $n\ge 1$, we have
\[
C_k(n) = O\left(\frac{n}{\vlog n}\right) \text{ and } \ 
i_k(kn+1) = O\left(\frac{n}{\vlog n}\right).
\]
\end{proposition}
\begin{proof}
To prove the statement, the idea is to split the summation 
range in $i_k(kn+1)$ into two parts, one with 
``small'' divisors of $kn+1$ and the other with ``large'' 
ones. More precisely, we separate the set of $d$'s dividing
$\ord{k}{kn+1}$ into the set of $d$'s that are less than or 
equal to $(kn+1)^{\alpha}$ and the set of those with are 
larger than $(kn+1)^{\alpha}$ for some ``small'' $\alpha$ 
that will be chosen later on.
Let us write
\[
\sum_{d | kn+1} \frac{\varphi(d)}{\ord{k}{d}} = S_1(n) + S_2(n),
\]
with 
\[
S_1(n) := \sum_{\substack{d | kn+1 \\ \ d \leq (kn+1)^{\alpha}}}
\frac{\varphi(d)}{\ord{k}{d}} 
\quad 
\text{ and }
\quad
S_2(n) := \sum_{\substack{d | kn+1 \\ d > (kn+1)^{\alpha}}} 
\frac{\varphi(d)}{\ord{k}{d}}.
\]
To obtain the desired result, we bound each term $S_1$ and $S_2$.

To obtain an upper bound for $S_1$, we use the fact that $\ord{k}{d} \ge 1$ and the observation that, if $d \leq (kn+1)^{\alpha}$, then $\varphi(d) \leq d \leq (kn+1)^{\alpha}$.
Therefore
\[
S_1(n) \leq (kn+1)^{\alpha} 
\left(\sum_{\substack{d | kn+1 \\ d \leq (kn+1)^{\alpha}}} 1 \right) \
\leq \ (kn+1)^{\alpha} \ \ \Card \{d\in [1,kn+1]\mid \textrm{$d$ divides $(kn+1)$}\}.
\]
Since the number of divisors of a positive integer $m$ satisfies the inequality
\[
\Card \{d\in [1,m]\mid\textrm{$d$ divides $m$}\} \leq 2 \sqrt{m}
\]
(we may group the divisors of $m$ pairwise, i.e., $d$ and $m/d$, and one of them is $\leq \sqrt{m}$), we have
\begin{align}
\label{eq: upper bound for S1}
    S_1(n) \leq 2(kn+1)^{\alpha+1/2} = O(n^{\alpha + 1/2}) = O\left(\frac{n}{\vlog n}\right)
\end{align}
for $\alpha < 1/2$.

To obtain an upper bound for $S_2$, we note that if $k^\ell \equiv 1 \bmod d$ with $\ell \neq 0$ and $d \neq 1$, then $k^\ell \geq d+1$, so $\ell\geq \vlog(d+1)/\vlog k$.
So, for the summation indices $d$ that appear in $S_2(n)$, we have 
\[
\ord{k}{d} \geq \frac{\vlog(d+1)}{\vlog k} \geq \frac{\alpha\vlog(kn+1)}{ \vlog k},
\]
since $d>(kn+1)^\alpha$.
It follows that
\[
    S_2(n) 
 \leq \left(\sum_{\substack{d | kn+1 \\ d > (kn+1)^{\alpha}}}
    \varphi(d)\right) \ O\left(\frac{1}{\vlog n} \right)
 \leq \left(\sum_{d | kn+1} \varphi(d)\right) \ O\left(\frac{1}{\vlog n} \right) 
 = (kn+1) \ O\left(\frac{1}{\vlog n} \right),
\]
where we used the well-known formula of~\cref{lem: famous formula with varphi} for the last equality.
We obtain 
\begin{align}
\label{eq: upper bound for S2}    
S_2(n) = O\left(\frac{n}{\vlog n}\right).
\end{align} 
Putting together~\cref{eq: upper bound for S1,eq: upper bound for S2} gives the desired asymptotics results.
\end{proof}

\begin{remark}
\label{rk: optmimality of bounds}
We note that~\cite[Corollaires page~8]{Allouche-1983-1984} already indicates the optimality of the previous bound for $(C_2(n))_{n\ge 1}$.
Here we show that, more generally, $n/\vlog n$ is the {\em right order of magnitude} for $i_k(kn+1)$.
Namely, let $\ell \geq 2$ be an integer, and set $n := k^{\ell} - k^{\ell-1} - 1$.
Thus, $kn+1 = (k^{\ell} - 1) (k - 1)$.
In particular, if $d$ divides $(k^{\ell} - 1)$, then
$d$ also divides $(kn+1)$, and furthermore
$\ord{k}{d} \leq \ord{k}{k^{\ell}-1} \leq \ell$.
Hence we get
\[
i_k(kn+1)=\sum_{d|(kn+1)}\frac{\varphi(d)}{\ord{k}{d}}
\geq \sum_{d | (k^{\ell}-1)} \frac{\varphi(d)}{\ord{k}{d}}
\geq \sum_{d | (k^{\ell}-1)} \frac{\varphi(d)}{\ell}
= \frac{k^{\ell} - 1}{\ell},
\]
where~\cref{lem: famous formula with varphi} is used 
for the last equality. Now, we are done, since when 
$\ell \to \infty$, which implies that $n \to \infty$, 
we have
\[
\frac{k^{\ell} - 1}{\ell} \sim
\frac{k^{\ell}}{\ell}
\quad \text{and} \quad
\frac{n}{\vlog n} \sim 
\left(\frac{k-1}{k \vlog k}\right) \frac{k^{\ell}}{\ell},
\]
where, as usual, the notation $f(x)\sim g(x)$ means that 
$f(x)/g(x)$ tends to $1$, when $x$ goes to infinity.


\end{remark}

\section{Unexpected occurrences of doubling modulo odd integers and other occurrences of the map \texorpdfstring{$i_k$}{i k}}\label{sec: occurrences}

In this section we propose to review some other occurrences of the doubling-modulo-an-odd-integer map, as well as occurrences of the map $i_k$.

\subsection{Toeplitz transforms and apwenian sequences}\label{sec:origin}
We begin this section by recalling the context for Toeplitz transforms and apwenian sequences that lead us to write this note.

$*$ In~\cite{Allouche-1983-1984} a notion of Toeplitz transform was studied, consisting of inserting a 
sequence with holes into itself and iterating the process 
until all holes have disappeared.
To illustrate the concept, we give the example of the regular paperfolding sequence.

\begin{example}
We start with the sequence $w = 0 \diamond 1 \diamond 0 \diamond 1 \diamond 0 \diamond 1 \diamond \cdots$, which consists of repeating the pattern $(0 \diamond 1 \diamond)$ infinitely many times.
The symbol $\diamond$ is the \emph{hole} and the next step is to insert the first sequence inside the holes, which yields the new sequence $0 0 1 \diamond 0 1 1 \diamond 0 0 \cdots$.
Repeating this procedure a second time gives the new sequence $0 0 1 0 0 1 1 \diamond 0 0 \cdots$, and so on and so forth.
The limit sequence obtained after infinitely many such insertions is called the \emph{Toeplitz transform} of the initial sequence $w$.
\end{example}

One of the questions addressed in~\cite{Allouche-1983-1984} consists in characterizing the binary periodic sequences having a periodic Toeplitz transform.
In particular, it is shown that the number of such sequences of period $(2n+1)$ is equal to $2^{C_2(n)+1}$, where we recall that $C_2(n)$ is the number of cycles of the $(2,n)$-perfect shuffle permutation $\sigma_n$ defined in~\cref{eq: permutation sigma n}.

\medskip

$*$ Now we turn to the concept of apwenian sequences and show how they relate to doubling modulo odd integers.
In~\cite{Guo-Han-Wu-2021}, a sequence $d= d_0 d_1 d_2 \cdots$ taking value over $\{+1,-1\}$ is said to be {\em apwenian} if, for every integer $n \geq 1$, its Hankel determinant $H_n(d)$ satisfies the congruence
\[
\frac{H_n(d)}{2^{n-1}} \equiv 1 \bmod 2.
\]
In their study of apwenian sequences that are fixed points of constant-length morphisms of the free monoid generated by $\{-1, +1\}$, the authors of~\cite{Guo-Han-Wu-2021} introduce the 
permutation $\tau_m$ defined in~\cref{eq: permutation tau m}.
More precisely, this permutation and its cycle decomposition allows them to count the number of apwenian sequences satisfying certain properties, see~\cite[Section~1.1]{Guo-Han-Wu-2021}.

\subsection{Dynamical systems}

Let $X$ be a compact metric space and $T\colon X\rightarrow X$ be a continuous transformation. We call $(X,T)$ a \emph{topological dynamical system}.
Let $Y$ be a closed nonempty subset of $X$ such that $T(Y)\subset Y$. Then the restriction of $T$ on $Y$ induces a topological dynamical system on $Y$, denoted $(Y,T)$ and called a \emph{subsystem} of $(X,T)$. Such a subsystem is called a \emph{minimal component} of $(X, T)$ if $(Y, T)$ is minimal, i.e., for each $y\in Y$, the orbit $\{T^my: m\geq 0\}$ is dense in $Y$. 
Let $k\geq 2$ and $n\geq 1$ be integers, and consider the topological dynamical system $(X,\sigma_{k,n})$, where the finite set $X=(\Z/(kn+1)\Z)\setminus \{0\}$ (endowed with the trivial metric) is compact.
The results in the present note can be reformulated as dynamical properties of $(X,\sigma_{k,n})$. For example,
$C_k(n)$ is just the number of minimal components contained in $(X, \sigma_{k,n})$.

More generally, let $a,b,m$ be fixed integers with $m\ge 2$, and consider the topological dynamical system $(\mathbb{Z}/m\mathbb{Z}, T_{a,b,m})$, where $T_{a,b,m}(x)=ax+b$, for each $x\in\mathbb{Z}/m\mathbb{Z}$. It is known that the system $(\mathbb{Z}/m\mathbb{Z}, T_{a,b,m})$
is minimal if and only if it is transitive, which is equivalent to saying that $b$ is coprime with $m$,
$a-1$ is a multiple of $p$ for every prime $p$ dividing $m$, and 
$a-1$ is a multiple of $4$ if $m$ is a multiple of $4$ (e.g., see~\cite[Theorem A, page 17] {Knuth-1998}. Also see
\cite{Larin-2002}). In particular, the system $(\Z/m\Z, T_{a,0,m})$ cannot be minimal. If $a$ is coprime with $m$, one can consider the multiplicative dynamical system $((\Z/m\Z)^{\times}, T_{a,0,m})$, and it is minimal if and only if $\ord{a}{m}=\varphi(m)$. 

By taking the projective limit of $\mathbb{Z}/p^\ell\mathbb{Z}\ (\ell\geq 1)$, one can extend the above results to $p$-adic dynamical systems. For example, the topological dynamical system $(\mathbb{Z}_p, T_{\alpha,\beta})$ is minimal if and only if $\alpha\in 1+p^{r_p}\mathbb{Z}_p$ and $\beta\in\mathbb{Z}_p^{\times}$ (see~\cite{Anashin-2006}, also see 
\cite{FLYZ-2007}),  and if the system is not minimal, it can be decomposed into minimal components, here $\mathbb{Z}_p$ is the ring of $p$-adic integers, $\mathbb{Z}_p^{\times }$ is the group of invertible elements in $\mathbb{Z}_p$, and we suppose that $\alpha,\beta\in\mathbb{Z}_p$, $r_2=2$, $r_p = 1$ for $p\geq 3$, and $T_{\alpha,\beta}(x)=\alpha x+\beta$, for each $x\in\mathbb{Z}_p$  (see~\cite{FLYZ-2007}). 
In the case that $\alpha\in \mathbb{Z}_p^{\times}$, one can also consider the multiplicative dynamical system $(\mathbb{Z}_p^{\times}, T_{\alpha,0})$, and discuss its minimality and its decomposition into minimal components (also see~\cite{FLYZ-2007}).

\subsection{The map \texorpdfstring{$i_k$}{i k}}

Note that other properties and interpretations of the map $i_k$ defined in~\cref{eq:ik} and other 
generalizations may be found, for example, in the 
works~\cite{Chasse1984,Chou-Shparlinski-2004,CRHII-2014,Deaconescu-2008,Larson-2019,Moree-2012,Moree-Sole-2005,Pomerance-Shparlinski-2010,Pomerance-Shparlinski-2018,Pomerance-Ulmer-2013,Rogers-1996,Sha-2011,Ulmer-2002,Ulmer-2014,Vasiga-Shallit-2004}.

\medskip

$*$ For example, in~\cite{Deaconescu-2008}, already cited 
above, the following result gives a characterization of 
Mersenne primes.

\begin{corollary}[{\cite[Corollaries~1 and~6]{Deaconescu-2008}}]
Let $k, n\ge 2$ be integers. Then 
\begin{align}
\label{eq: inequality for Mersenne primes}
 n i_k(k^n-1) = 
\sum_{d|n} \varphi\left(\frac{n}{d}\right) (k^d - 1) \geq 
\sum_{d|n} \frac{n}{d} \varphi(k^d-1).   
\end{align}
In particular, $k^n - 1$ is a (Mersenne)
prime if and only if~\cref{eq: inequality for Mersenne primes} is an equality.
\end{corollary}

\medskip

$*$ Another occurrence of the map $i_k$ (and several of its 
properties) can be found in~\cite{Moree-Sole-2005} 
where {\em very odd sequences} are studied.
Let $(a_i)_{1\le i\le n}$ be a sequence of $n$ integers in $\{0, 1\}$ (otherwise stated, a binary sequence of length $n$) and set $A_\ell = \sum_{1 \leq i \leq n-\ell} a_ia_{i+\ell}$, for $\ell\in [0,n-1]$.
The sequence $(a_i)_{1\le i\le n}$ is called a \emph{very odd sequence} if $A_\ell$ is odd for each 
$\ell\in [0,n-1]$.

\begin{proposition}[{\cite[Proposition~2]{Moree-Sole-2005}}]
For each integer $n\ge 1$, the number $S(n)$ of very odd sequences of length $n$ is given by 
\[
S(n) =
\begin{cases}
    0, \ &\text{if $\ord{2}{2n-1}$ is even;} \\
    \sqrt{2}^{i_2(2n-1) - 1}, 
    &\text{otherwise.}
\end{cases}
\]
\end{proposition}

We have already cited in~\cref{rk:ik-irreducible-factors}
a statement that can be found, e.g., 
in~\cite[Lemma 5]{Moree-Sole-2005}, namely that, if $q$ is the order of a finite field, then $i_q(qn+1)$ is the number of distinct irreducible 
factors of the polynomial $X^{qn+1} -1$ on $\mathbb{F}_q$. 
Actually one also finds the following result.

\begin{proposition}[{\cite[Remark, page~224]{Moree-Sole-2005}}]
Let $p$ be a prime number. 
Then for each integer $n$ with $p \nmid n$, $i_p(n)$ counts the total number of prime 
ideals that the ideal $(p)$ factorizes in, in all the cyclotomic 
subfields of ${\mathbb Q}(\zeta_n)$, where $\zeta_n$
is a primitive $n$-th root of unity.
\end{proposition}

\medskip

$*$ In the papers~\cite{Ulmer-2002,Pomerance-Shparlinski-2010} 
the quantity $i_k$ appears in the study of the
parametric family of elliptic curves
${\mathbf E}_d :$ \ $y^2 + xy = x^3 - t^d$, over the 
function field ${\mathbb F}_q (t)$, where $d$ is a 
positive integer. In the same vein, one can look 
at~\cite[Corollary~5.3]{Ulmer-2014} and 
\cite[Theorem~2.2]{CRHII-2014} (also see~\cite{Pomerance-Ulmer-2013}).

\medskip

$*$ In~\cite[page~320]{Rogers-1996} one finds (a slight 
variation on) $i_k$ with the following statement.
For a finite set $G$ and a map $f \colon G \to G$ from $G$ to
itself, let $\graph_f(G)$ be the directed graph whose 
vertices are the elements of $G$ with a directed 
edge from $x$ to $f(x)$ for every $x \in G$.

\begin{proposition}[{\cite[page~320]{Rogers-1996}}]
Let $k,m,n$ be positive integers with $m$ odd and $n = 2^k m$.
Then the number of cycles of $\graph_f({\mathbb Z}/n{\mathbb Z})$
relative to $f(x):= x^2$ is equal to $i_k(m)$.
\end{proposition}

Also see~\cite[Section~5]{Rogers-1996} for other results
with $i_k$ in this context. For results about iterations 
of maps and counting cycles, see~\cite{Vasiga-Shallit-2004} 
and its bibliography (in particular~\cite{Chasse1984}), also
see~\cite[Theorem~1]{Chou-Shparlinski-2004}; for more recent 
occurrences of (variants of) $i_k$, e.g., see 
\cite[Proposition~3.1~(6)]{Sha-2011}, \cite{Pomerance-Shparlinski-2018} and~\cite{Larson-2019}.

\subsection{The Luhn algorithm}

There is one more example of doubling modulo an odd integer, somehow ``simple'' but widely used nowadays, namely the {\em Luhn algorithm} or \emph{Luhn formula} originally described by Luhn in 1960~\cite{Luhn}.
In this particular case, we write numbers in base $10$ but we double some digits modulo $9$ and the goal is to distinguish valid numbers from incorrect ones and to detect errors in writing numbers.
More precisely, with each integer $n$ written in base $10$, say $n = a_d a_{d-1} \dots a_0$ with $a_i\in[0,9]$ for every $i$, is associated the integer $n' = a'_d a'_{d-1} \dots a'_0$ obtained with the following rules: for $j = d, d-2, d-4, \ldots$, we set $a'_j = a_j$, and for $j = d-1, d-3, d-5, \ldots$, we set $a'_j = 2a_j \bmod 9$, where $m \bmod 9$ is the integer congruent to $m$ modulo $9$ that belongs to $[0, 8]$.
Then there are two possible uses: either the sum $S := \sum_j a'_j$ of new digits is ``checked'' to be divisible by $10$, or $S$ is used to generate a {\em key} $x$, where $x$ is defined by 
$x \in [0, 9]$ and $S+x$ divisible by $10$.

\subsection{Card-shuffling}

Shuffling cards has always been somewhere between magic and mathematics.
The mathematical study of card shuffle goes back at least to Monge~\cite{Monge-1773} in the 18th century. 
A particular card shuffling is called \emph{le battement de Monge} in French and \emph{Monge card shuffle} in English 
(see \cite[Problem~15, pages~214--222]{Bachet} 
and~\cite{Bourget}; also see 
\cite{Roubaud-2006, Roubaud-2014}), it is associated 
with the permutation, called the {\em Monge permutation} and defined by
\[
\left(
\begin{matrix}
1 &2 &3  &\ldots &n \ &n + 1 &\ldots &2n-2 \ &2n-1 \ &2n \\
2n\ &2n-2 \ &2n-4 &\ldots &2 &1 &\ldots &2n-5 \ &2n-3 \ &2n-1 \\
\end{matrix}
\right)
\]
for some integer $n\ge 1$.
Observe that the latter permutation is similar to 
(but different from) the permutation 
$\sigma_n$ from~\cref{eq: permutation sigma n}.
Also note that a variant of the Monge card shuffle occurs 
in~\cite[Section~4]{Roberts-1969} in relation with 
quadratic residues and primitive roots.

\smallskip

A particular card-shuffling and the mathematics behind
were explicitly described in several papers. Let us cite 
\cite{Diaconis-Graham-Kantor-1983,Ellis-Fan-Shallit-2002,Morris-Hartwig-1976}. 
For example one can read in~\cite{Ellis-Krahn-Fan-2000}
the following comment about the permutation $x \mapsto 2x \bmod (n+1)$ over $[1,n]$, 
called {\em the perfect shuffle permutation of order n}:

\begin{quote}
{\em The name is taken from
that method of shuffling a deck of cards that cuts the
deck into halves and then interleaves the two halves
perfectly.}
\end{quote}
Note that in the book~\cite{Diaconis-Graham-2019}, one
can find (on page~166) the map $x \to 2x \bmod 11$.
Among several other papers on the subject, let us 
only cite the four papers~\cite{Medvedoff-Morrison-1987} 
(where the heaps are shuffled according to some 
permutation, and where---on page~6---the permutation 
$\sigma_{k,n}$ is denoted by $w_\text{rev}$), \cite{Ramnath-Scully, Quintero-2010}, and \cite{Lachal-2010}.

\subsection{Juggling}

First we mention the nice book~\cite{Polster-2003} that
tells everything one always wanted to know about mathematics 
and juggling. In particular, a reformulation of the statement 
on top of~\cite[page~36]{Polster-2003}, where one takes 
$q=3$ and $p=2n+1$ (with $n\not\equiv 1 \bmod 3$) gives the 
following result.

\begin{proposition}
Let $n$ be a positive integer with $n\not\equiv 1 \bmod 3$.
Then the permutation $\sigma_n$ gives a {\em magic juggling sequence} when extending it to a map on $[0,2n]$ by setting $\sigma_n(0) := 0$.
\end{proposition}

\subsection{Bell-ringing}

Bell-ringing is both an art and a science.
For example, the abstract of a lecture given at Gresham College 
by Hart and entitled ``The Mathematics of Bell Ringing''~\cite{Hart-2021} begins with the following lines: 

\begin{quote}
{\em This lecture will look at change ringing, which is 
ringing a series of tuned bells (as you might find in the 
bell tower of a church) in a particular sequence, and this 
has exciting mathematical properties},
\end{quote}

\noindent
while one can read on~\cite[page~116]{Roaf-White-2003}:

\begin{quote}
{\em Compositions in ringing are designed to include 
musically attractive sequences, which are usually based 
on sequences running up or down the scale (``roll-ups''), 
sometimes with single notes omitted to produce slightly 
larger intervals of pitch. }
\end{quote}

For more on bell-ringing, also 
see~\cite{Stedman-1677, Jaulin-1977,Roaf-White-2003}, 
\cite[Chapter~6]{Polster-2003}, the first page of~\cite{Mayfield}, and the nice video~\cite{Hart-2021}.

Now, looking at~\cite[page~116]{Roaf-White-2003} again,
one can read:

\begin{quote}
{\em Taking every other note of the scale, and running 
through the scale twice so as to include all the bells, 
produces $135246$ on six bells, or $13572468$ on eight. 
This is the best known of these favourite rows; it is 
called {\em Queens}, because a Queen of England is said 
to have commented on how nice the bells sounded when she 
heard it being rung. 

Reversing the first half of Queens on six bells produces 
Whittingtons: $531246$; when heard by Dick Whittington 
leaving London as ``Turn again, Whittington; Lord May'r of 
London", it persuaded him to return and, eventually, become 
Lord Mayor.}
\end{quote}

\noindent
The reader has certainly ``recognized'' the permutations
$135246$ and $13572468$ that have the same flavor as our $\sigma_3=(2 \; 4 \; 6 \; 1 \; 3 \; 5)$ and $\sigma_4= (2 \; 4 \; 6 \; 8 \; 1 \; 3 \; 5 \; 7)$ as above with~\cref{eq: permutation sigma n} for $n=3,4$.
Actually there is more than a common flavor as shown below.

\begin{lemma}\label{thm}
Let $n$ be a positive integer.
Define the Queens permutation $\theta_n$ on $[1, 2n]$ by
\begin{equation}
\label{eq: theta n}
\theta_n :=
\left(
\begin{matrix}
1 \ &2 &\cdots &n \ &n+1 \ &n+2 \ &\cdots &2n \\
1 \ &3 &\cdots &2n-1 \ &2 \ &4 &\cdots &2n \\
\end{matrix}
\right).
\end{equation}
Then, the restriction of $\theta_n$ to $[2, 2n-1]$ is
the same, up to notation, as the permutation $\sigma_{2,n-1}$.
\end{lemma}

\begin{proof}
Note that excluding $1$ and $2n$ is somehow ``natural'' 
since the {\em Queens} permutation maps the first bell 
to the first and the last bell to the last.
Now the restriction of $\theta_n$ to $[2, 2n-1]$ gives
\begin{equation*}
\theta_n|_{[2,2n-1]} :=
\left(
\begin{matrix}
2 & 3 &\cdots &n \ &n+1 \ &n+2 \ &\cdots &2n-1 \\
3 & 5 &\cdots &2n-1 \ &2 \ &4 &\cdots &2n-2
\end{matrix}
\right).
\end{equation*}
After the change of variable $k \mapsto k-1$, we obtain the
permutation
\begin{equation*}
\left(
\begin{matrix}
1 & 2 &\cdots &n-1 \ &n \ &n+1 \ &\cdots &2n-2 \\
2 & 4 &\cdots &2n-2 \ &1 \ &3 &\cdots &2n-3 \\
\end{matrix}
\right),
\end{equation*}
which is exactly the permutation $\sigma_{2,n-1}$ (defined on $[1,2n-2]$) from~\cref{eq: permutation sigma n}.
\end{proof}

\subsection{Poetry}\label{poetry}

A variant of the permutation $\sigma_{2,3}$ can be found in 
various papers on poetry and literature with connections to 
mathematics; (e.g., see
\cite{Roubaud-1969}, \cite[page~65]{Audin-2007}, 
\cite{Audin2,Bringer-1969,Dumas-2008,Queneau-1967,Tavera-1967}; 
also see the paper~\cite{Audin-2013} and the 
book~\cite{Lartigue1994}). This permutation is given by
\[
\nu_6 = 
\left(
\begin{matrix}
1 \ &2 \ &3 \ &4 \ &5 \ &6 \\
2 \ &4 \ &6 \ &5 \ &3 \ &1 \\
\end{matrix}
\right)
\]
and is related to the so-called \emph{sextine} in French 
and \emph{sestina} in English.
(Note that this permutation $\nu_6$ is the inverse of the 
Queneau-Daniel permutation $\mu_6$ defined 
in~\cref{sec: first interpretation}, see \cref{lev} below.)
In this form of poetry, the last words of a \emph{stanza} of 
six lines are permuted according to the previous permutation 
and used as ending words in the next stanza of six lines.
It was invented by the 12th century poet Arnaut Daniel (for his poem 
and an English translation, e.g., see~\cite{Saclolo-2011}).
Over time, cases with $n$ elements instead of $6$, as well
as variations around these permutations were 
studied, both on the literary and the mathematical sides.
In particular, we mention the book \cite{Roubaud-2014}, 
with its chapter called {\em Le battement de Monge} 
\cite[pages~145--163]{Roubaud-2014} (for more on {\em tropical 
partitions} described there, see~\cite{Guilbaud-1972}); also
see~\cite{BL-BV-U}.

\begin{remark}\label{lev}
We cannot resist to mention that the permutation $\nu_6$ above resembles
very much permutations in the family studied by L\'evy in 
\cite[Chapter~1]{Levy-1950}. Namely, this variant $\nu_6$ of 
$\sigma_{2,3}$ is the case $n=6$ of the permutation $\nu_n$ defined 
for $x \in [1,n]$ by
\[
\nu_n(x) = 
\begin{cases}
    2x, &\text{if } 2x \leq n; \\
    2n+1-2x, &\text{otherwise,}
\end{cases}
\]
while the permutation $\rho_n$ studied by L\'evy is defined 
in~\cite[Chapter~1]{Levy-1950} for $x\in[1,n]$ by
\[
\rho_n(x)=
\begin{cases}
2x-1, &\text{if } 2x \leq n+1; \\
2n+2-2x, &\text{otherwise.}
\end{cases}
\]
Recall that $\sigma_{2,n}$ can also be defined for $x\in[1,2n]$ by
\[
\sigma_{2,n}(x) =
\begin{cases}
2x, &\text{if} \ x \leq n; \\
2x - 1 - 2n, &\text{otherwise.}
\end{cases}
\]
Note that easy identities can be proved, e.g.,
$\forall k \in [2,n]$, we have
$\rho_n(k) = \nu_{n-1}(k-1) +1$, or, as already mentioned
in \cref{sec: first interpretation}, $\nu_n$ is 
the inverse of the Queneau-Daniel permutation $\mu_n$ 
defined on $[1, n]$ by $\mu_n(k) = k/2$ if $k$ is even, 
and $\mu_n(k) = n - (k-1)/2$ if $k$ is odd, 
\cite{Bringer-1969}. 

\bigskip

Thus we see that all these permutations share a common flavor.
It is interesting to read the opinion of L\'evy on mathematical 
problems of this sort (see~\cite[page~151]{Levy-1970}):
\begin{quote}
{\em Je ne m'exag\`ere pas l'importance du probl\`eme, qui se
rattache \`a l’analyse combinatoire et \`a la th\'eorie des 
groupes, dont je vais parler maintenant. Aux yeux d'un 
math\'ematicien plus jeune que je ne le suis, cela ne peut 
\^etre qu'une toute petite chose. Mais il me semble qu'il 
peut y avoir int\'er\^et \`a montrer que, dans l'esprit des 
math\'ematiques classiques, il y a encore de jolis 
probl\`emes \`a traiter. En outre, la mani\`ere dont je fus 
conduit \`a poser celui dont je vais parler est assez 
amusante {\rm [...]}.}
\end{quote}
\end{remark}

\subsection{Musical composition}

One can ask whether the permutations studied here have been
used in musical composition beyond bell-ringing.
The answer is of course positive.
First, let us mention the must-read~\cite{Amiot-2008} and 
its bibliography. The interested reader can also look 
at~\cite{Adler-Allouche-2022} where the authors
are interested in which permutations $\sigma_n$ satisfy an 
extra property.
One can in particular listen to the piece \emph{Science Fictions} for two pianos by Adler (see~\cite{Adler-2021}). 

The curious reader can also be interested in the use of
sestinas (see~\cref{poetry}) in music: we will only cite
the {\it Sestina} of Monteverdi; e.g., listen to
\url{https://www.youtube.com/playlist?list=PL2k8ekJXk4nVHGs3glCLvc9EwD_GUEpQ-}.

\section{Conclusion and open directions}

In this paper, we have proven in~\cref{sec: second interpretation} the equality between the 
quantities~\cref{expression1,expression2} by considering a 
generalized permutation of $[1,kn]$ that can be written as 
$x\mapsto k x \bmod{(kn+1)}$, where $k,n$ are positive 
integers.
In addition to studying various properties of this 
permutation in~\cref{sec: first interpretation,sec: asymptotics}, we have listed some of its many occurrences 
in the literature in~\cref{sec: occurrences}.
Going even beyond, the reader might wonder whether other 
generalizations are possible. 
We point out the map $x\mapsto kx \bmod{(kn+j)}$ where 
$j,k,n$ are integers.
Observe that when $k,j$ are not coprime, then the map is 
not a bijection, e.g., $x\mapsto 2x \bmod{(2n+4)}$ maps,
for $n=2$, $(1, 2, 3, 4, 5, 6)$ to $(2, 4, 6, 8, 0, 2)$.
In the case where $k,j$ are coprime, we think 
that the results of this paper can be adapted, for which 
we leave the details to the reader, especially the precise 
description of the cycles of~\cref{thm: precise description of cycles} by Ellis, Fan, and Shallit.

\subsection*{Acknowledgments}
The first author proposed the equality stated  in~\cref{cor: equality between our two quantities} as one of the subjects of the ``Concours SMF Junior 2022'' (organized by the French Mathematical Society).
He thanks the  participants for their work and suggestions.
The three authors want to warmly thank 
Mich\`ele Audin, R\'emy Bellenger, Pieter Moree, Olivier
Salon, and Jeffrey Shallit for highlighting discussions and providing several useful references.


%
%
%
\bibliographystyle{splncs04}
\bibliography{all-stip-yao-original.bib}

\end{document}